\documentclass[11pt,a4paper,oneside,english]{article}
\usepackage{lmodern}
\usepackage[T1]{fontenc}
\usepackage[utf8]{inputenc}
\usepackage[english]{babel}
\usepackage{amscd,latexsym, mathrsfs, epsfig}
\usepackage{amsmath,amsfonts,amssymb,amsthm}
\usepackage{mathdots}
\usepackage{upgreek}
\usepackage[all]{xy}
\usepackage{graphicx,color}
\makeatletter \@addtoreset{equation}{section} \makeatother

\newtheorem{thm}{Theorem}[section]
\newtheorem{prop}[thm]{Proposition}
\newtheorem{lem}[thm]{Lemma}

\newtheorem{pb}[thm]{Problem}
\newtheorem*{thm*}{Theorem}
\theoremstyle{definition}
\newtheorem{definition}[thm]{Definition}
\newtheorem{exm}[thm]{Example}
\newtheorem{rmk}{Remark}
\newtheorem*{pb*}{Problem}
\usepackage{hyperref}
\newcommand\C{\mathbb C}
\newcommand\R{\mathbb R}
\newcommand\Z{\mathbb Z}
\newcommand\N{\mathbb N}
\newcommand\Q{\mathbb Q}
\newcommand\ii{i}
\newcommand{\sgn}{\operatorname{sgn}}
\renewcommand{\Re}{\operatorname{Re}} %real and imaginery part
\renewcommand{\Im}{\operatorname{Im}}
\newcommand{\Rp}{\mathbb{R}_{>0}}
\newcommand{\Rm}{\R_{<0}}
\newcommand\PP{\mathbb P} %projective space
\newcommand{\Hom}{\operatorname{Hom}}
\newcommand{\calH}{\mathcal{H}}
\renewcommand{\H}{\mathbb{H}}
\newcommand{\stokes}{\mathcal{S}} %stokes factor/multiplier

\newcommand{\XX}{\mathfrak{X}}
\newcommand{\XXp}{\hat \XX}
\newcommand\dt{\operatorname{DT}}
\newcommand{\pow}[1]{[\![ {#1} ]\!]}
\newcommand{\bfs}{\mathbf{s}}
\newcommand{\torus}{\mathbb{T}}
\renewcommand{\th}{\theta}
\newcommand{\thdual}{\theta^{\vee}}
\newcommand{\gdual}{\gamma^{\vee}}
\newcommand{\e}{\epsilon}
\newcommand{\holosect}{\Delta}
\newcommand{\coef}[2]{ a( {#1} )_{ {#2} } }
\newcommand{\Ham}{\operatorname{Ham}} 
\newcommand{\Gactive}{\Gamma^+_{\Omega}}
\newcommand{\Gactiver}{\Gamma_r^{\Omega}}
\newcommand{\mmg}[2]{ \Lambda\left({#1},{#2}\right) } %multivalued modified gamma
\newcommand{\mg}[2]{ \Lambda_{{#1}}\left({#2}\right) }
 
\newcommand{\mgi}[2]{ \Lambda_{ 1-{#1} }^{-1} \left( {#2} \right) } 
\newcommand\bra{\langle}
\newcommand\ket{\rangle}
\newcommand{\del}{\partial}
\newcommand{\deld}[1]{\frac{\del}{\del {#1} }}
\newcommand{\frap}{\frac{1}{2\pi\ii}}
\newcommand\Id{\operatorname{Id}}
\newcommand\Aut{\operatorname{Aut}}

\newcommand{\Li}{\operatorname{Li}} %polylog

\newcommand{\G}{\mathcal{G}}

\newcommand{\gl}{\mathfrak{gl}}
\newcommand\CC{\mathcal{C}}
\usepackage{tikz}
\usetikzlibrary{shapes.misc}
\tikzset{cross/.style={cross out, draw=black, minimum size=2*(#1-\pgflinewidth), inner sep=0pt, outer sep=0pt},cross/.default={1pt}}
\usepackage[all]{xy}
\usepackage{graphicx,color}
\usepackage{makeidx}
\makeindex

\title{A Riemann-Hilbert problem for uncoupled BPS structures}
\author{Anna Barbieri}
\date{}
\begin{document}

\maketitle
\begin{abstract}
\noindent We study the Riemann-Hilbert problem attached to an uncoupled BPS structure proposed by Bridgeland in \cite{bridg_RHpb}. We show that it has \lq\lq essentially\rq\rq\ unique meromorphic solutions  given by a product of Gamma functions. We reconstruct the corresponding connection.
\end{abstract}

\section*{Introduction}
This paper studies the instance of Riemann-Hilbert problem proposed by Bridgeland in \cite{bridg_RHpb} for uncoupled BPS structures. It is stated in terms of complex-valued functions and it is solved in \cite{bridg_RHpb} for a fixed value of a certain parameter. We show that for any value of that parameter the solution is a pair of meromorphic functions expressed explicitly as a product of Gamma functions. An integral representation of the solution is used to reconstruct the corresponding connection.

The same class of Riemann-Hilbert problems was considered by Filippini, Garcia-Fernandez and Stoppa in \cite{FGS}, motivated by the physics work \cite{GMN}. Their solution takes values in the automorphism group of an algebraic torus. The contexts of \cite{FGS} and \cite{bridg_RHpb} are slightly different, and comparing the two articles might require some efforts. We show in which sense and to which extent the two problems (and the corresponding solutions) are related, and we propose a new way to express the solution. In turn, this is analogous to the \lq\lq conformal limit\rq\rq\ of coordinates for the moduli spaces of $\mathcal N=2$ four-dimensional gauge theories compactified on a circle, presented by Gaiotto in \cite{gaiotto}, the main difference being that we consider coordinates on a complex torus. Our discussion about the solutions will allow us to consider also the \lq\lq quantized\rq\rq\ version of this problem, \cite{qRH}. 

Riemann-Hilbert problems are inverse problems in the theory of differential equations. They classically consist in seeking a piecewise holomorphic function on $\C^*$ with values in a Lie group, with prescribed behaviour near the origin and jumping discontinuities along a real-codimension $1$ boundary, \cite{fokas_painleve}. A BPS structure is an instance of the stability data defined by Kontsevich and Soibelman \cite{KS} and contains the information from the unrefined Donaldson-Thomas theory of dimension three Calabi-Yau categories. It defines naturally a Riemann-Hilbert problem with values in the automorphism group of an algebraic torus that, in some nice cases, can be traslated into a scalar problem, \cite{bridg_RHpb}. Riemann-Hilbert problems for BPS structures are relevant in some attemps of defining a Frobenius manifold type structure from Donaldson-Thomas theory, \cite{BS,SS}.

In the rest of the introduction we illustrate the content of the paper.
\smallskip\\
\textbf{BPS structures.} In the first Section we briefly recall some notions about integral BPS structures $(\Gamma,Z,\Omega)$ and the associated twisted torus $\torus$. They are defined by a finite rank lattice $\Gamma$ with a pairing $\bra-,-\ket$, a homomorphism $Z:\Gamma\to\C$, and a map $\Omega:\Gamma\to \Z$. The twisted torus is the space 
	\[\torus := \big\{\xi:\Gamma\to\C^*:\xi(\alpha+\beta)=(-1)^{\bra\alpha,\beta\ket}\xi(\alpha)\xi(\beta) \big\},
	\]
with characters $x_{\gamma}$, $\gamma\in\Gamma$, acting on it as $x_{\gamma}(\xi)=\xi(\gamma)$. We restrict to a class of BPS structures called uncoupled. This is the analogue of the physics terminology of \lq\lq mutually local\rq\rq\ BPS structures. We impose moreover finiteness and convergence hypotheses.

The basic example of an uncoupled BPS structure is the \lq\lq doubled $A_1$ BPS structure\rq\rq , defined by a lattice $\Gamma=\Z\cdot\alpha\oplus\Z\cdot\alpha^{\vee}$, a central charge $Z\in\Hom\left(\Gamma,\C\right)$ with $Z(\alpha^{\vee})=0$, and a symmetric map $\Omega:\Gamma\to\Z$ with $\Omega(\pm\alpha)=1$ and vanishing otherwise.
\smallskip\\
\noindent \textbf{A Riemann-Hilbert problem.} In the second section a Riemann-Hilbert problem for uncoupled BPS structures is introduced. To the active rays 
\[\ell_{\gamma}=Z(\gamma)\Rp\subset \C^*,\ \text{ for } \Omega(\gamma)\neq 0,\] are attached transforms $\stokes(\ell_{\gamma})$ of the torus $\torus$. Let $\Sigma\subset\C^*$ be the union of active rays. We are interested in finding a sectionally holomorphic map $$\Psi: \C^*\setminus\Sigma \to \Aut\left(\torus\right)$$ with discontinuities on each component $\ell_{\gamma}$ of $\Sigma$ given by the composition with $\stokes(\ell_{\gamma})$ (jumping condition), asymptotic behaviour near the origin 
\begin{equation*}
\lim_{t\to 0}\Psi(t) \circ \exp(Z/t) = \Id,
\end{equation*}
and algebraic behaviour at infinity. The uniqueness of the solution depends on the possibility of extending the restriction of $\Psi$ to any sector bounded by two consecutive active rays over its edges. In the uncoupled case, the $\Aut(\torus)-$valued Riemann-Hilbert problem can be turned into a scalar problem (that is with complex values) by fixing a point $\xi\in\torus$, evaluating $\Psi$ in $\xi$ and then applying it to a point $\beta\in\Gamma$. We obtain the following diagram that allows for a complex-analytical approach.
\begin{equation*}
\xymatrix{
\C^*\setminus \Sigma \ar[r]^-{\Psi} & \Aut(\torus) \ar[r]^-{ev_{\xi}} & \torus\ar[r]^-{ev_{\beta}} & \C \ar@{^{(}->}[d] \\
\quad\quad\quad\Delta\subset\H_{\Delta} \ar@{^{(}->}@<-0.6ex>[u] \ar[rrr]^-{Y_{\beta,r}} &&&\C\PP^1
}
\end{equation*}
Here $\Delta$ is a holomorphicity sector for a solution $\Psi$ and $\H_{\Delta}$ is any open half-plane centred in a non-active ray $r$ contained in $\Delta$. For non-active rays $r$, we seek for complex valued functions $Y_{\beta,r}$, that can be compared in the common domain of definition. This is the approach of \cite{bridg_RHpb}, where $Y_{\beta,r}$ are required to be holomorphic and never-vanishing. In fact the scalar problem as stated in \cite{bridg_RHpb}  does not always admit solutions (Proposition \ref{prop_holoRH_nosol} below), but it can be reformulated and solved in terms of meromorphic functions (Problem \ref{meroRHpb}) below. This is the scalar counter-part of the $\Aut(\torus)-$valued Riemann-Hilbert problem. It has \lq\lq essentially\rq\rq\ at most one solution, i.e.\ unique up to the choice of vanishing order of a finite number of points. 
\smallskip\\
\noindent\textbf{Solutions and Hamiltonian vector field.} The solution to the scalar Riemann-Hilbert problem and the corresponding connection are considered in sections \ref{sec_solutions} and \ref{sec_hamilt}, which are based on section \ref{sec_complex_analysis}, where we develope the analytical background. We introduce the function 
\begin{equation*}
	\mg{x}{y}:= \frac{\Gamma(x+y) \cdot e^y}{y^{x+y-\frac{1}{2}}\cdot \sqrt{2\pi}}.
	\end{equation*}
It is a modification in two variables of the Gamma function and its relevant properties are listed in Lemma \ref{lem_Sol_jump} and Theorem \ref{thmSolInt}.
	
In section \ref{sec_solutions} we prove that, for every $\xi\in\torus$, there exist non-trivial meromorphic functions $\left\{Y_{\beta,r}\right\}_{\beta,r}$ solving the meromorphic Riemann-Hilbert problem for uncoupled BPS structures (Theorem \ref{thm_sol_meroRH_finite}). In the doubled $A_1$ case, for example, for any $\xi\in\torus$, $Y_{\alpha,r}\equiv 1$ and the solution is encoded in two meromorphic functions $Y_{\pm}:\C^*\setminus\pm\ii Z(\alpha)\Rp \to \C\PP^1$, obtained by gluing together $Y_{\alpha^\vee,r}$ as $r$ lies on one or the other side of $\pm\ell$.   We have
	\begin{equation*}
	Y_-(t)=\mg{\frac{\th}{2\pi\ii}}{-\frac{Z(\alpha)}{2\pi\ii t}} \quad \text{and} \quad Y_+(t)={\mgi{\frac{\th}{2\pi\ii}}{\frac{Z(\alpha)}{2\pi\ii t}}},
	\end{equation*}
where $\th:=\ln\xi(\alpha)$, for a chosen branch of the logarithm. These obviously coincide with the result in \cite{bridg_RHpb} when $\th=0$, and are very closed related to \cite[{Eq.\ 3.10}]{gaiotto}.

The inverse problem is considered in Section \ref{sec_hamilt} for a BPS structure with trivial pairing. \lq\lq Doubling\rq\rq\ the construction, one has that $\torus$ has a symplectic structure. From $\{Y_{\beta,r}\}_{\beta,r}$ we deduce $\Psi$ and compute a connection $\nabla$ on the trivial $\Aut(\torus)$-bundle over $\C\PP^1$, such that $\nabla\Psi=0$. Say $$F=\frac{1}{2\pi i} \sum_{\gamma\in\Gamma\setminus\{0\}} \Omega(\gamma)\Li_2\left(x_{\gamma}\right),$$ 
a function $\torus\to\C$. $\nabla$ has the form $\nabla=d-\left(\frac{Z}{t^2}-\frac{\Ham_{F}}{t}\right) d t$. Analogous computations allows to define a similar connection for any uncoupled BPS structure.

\subsection*{Acknowledgements} 
The author is grateful to Tom Bridgeland %for pointing out the problem and for many enlightening discussions and remarks, 
for many interesting discussions. %, and to the anonymous referee for useful comments. 
Thanks are also due to Dylan Allegretti and Jacopo Stoppa for their comments on the preliminary version. This research was supported by the European Research Council, through the project ERC-AdG StabilityDTCluster.

\section{BPS structures and notation}\label{sec_background}

We briefly recall the notion of a BPS structure. The aim of this section is to fix the notation for the rest of the article.  Most of the definitions recalled in the following are from  \cite{bridg_RHpb}, where it is possible to find a wider explanation of the mentioned objects. 

\begin{definition}\label{def_BPS} A \emph{BPS structure} $(\Gamma,Z,\Omega)$ of rank $n$ is the datum of a finite rank lattice $\Gamma\simeq \Z^{\oplus n}$ (the \emph{charge lattice}) endowed with an \emph{intersection form}, that is an integral, bilinear and skew-symmetric pairing 
	\begin{equation*}
	\bra -,-\ket :\Gamma\times\Gamma\to\Z,
	\end{equation*}
a homomorphism $Z:\Gamma\to \C$, and a map of sets $\Omega:\Gamma\to\Q$, such that
\begin{itemize}
\item[(i)] $\Omega$ is symmetric, i.e.\ $\Omega(-\alpha)=\Omega(\alpha)$ for all $\alpha\in\Gamma$, and 
\item[(ii)] there is a uniform constant $C>0$ such that, for some fixed norm $||\cdot||$ in $\Gamma\otimes\R$, $|Z(\alpha)|>C\cdot||\alpha||$ for all $\alpha$ with $\Omega(\alpha)\neq 0$.
\end{itemize} 
$Z$ is called a \emph{central charge} and $\Omega$ is the \emph{BPS spectrum}.

\end{definition}
We denote by $\Gactive$ the subset of $\Gamma$
	\begin{equation}\label{Gactive}
	\Gactive:=\left\{\gamma\in\Gamma\setminus\{0\} : \Omega(\gamma)\neq 0 \text{ and } Z(\gamma)\in\calH^+\right\},
	\end{equation}
where $\calH^+$ is the upper half-plane together with the negative real line $$\calH^+=\{z\in\C^* : 0< \arg(z) \leq \pi\}.$$

\begin{definition} An \emph{active class} is a point $\gamma\in\Gamma$ such that $\Omega(\gamma)\neq 0$. For every active class, we introduce an \emph{active ray} $\ell_{\gamma}:=Z(\gamma)\Rp\subset\C^*$. An active ray is sometimes referred to as a \emph{BPS ray}. A ray $r\subset\C^*$ which is not active is said to be \emph{generic}.
\end{definition}
\begin{definition} A \emph{null vector} is a point $\alpha\in\Gamma$ such that $\bra\alpha,\beta\ket= 0 $ for all $\beta\in\Gamma$.
\end{definition}

\begin{definition}\label{conditions} A BPS structure is said to be 
\begin{itemize}
\item \emph{generic} if for any two active classes $\gamma_1,\gamma_2$, the existence of a real non-zero $\lambda$ such that $Z(\gamma_1)=\lambda Z(\gamma_2)$ implies that $\bra\gamma_1,\gamma_2\ket=0$;
\item \emph{uncoupled}, if $\bra \gamma,\delta\ket = 0 $ for all active classes $\gamma,\delta$;
\item \emph{integral} if $\Omega$ takes values in $\Z$; and 
\item \emph{convergent} if there exists $\lambda>0$ such that $\sum_{\gamma\in\Gamma}|\Omega(\gamma)|\cdot e^{-\lambda|Z(\gamma)|}<\infty$.
\end{itemize}
\end{definition}
\noindent In particular an uncoupled BPS structure is generic. 

In this article we will mostly assume that a BPS structure is ray-finite, i.e.\ there are finitely many active rays, or finite, i.e.\ there are only finitely many active classes $\gamma\in\Gamma$.

\subsubsection*{Twisted torus}
The algebra $\C[\Gamma]$ of formal elements $x_{\alpha}$, $\alpha\in\Gamma$, comes endowed with a commutative product $\cdot$
	\begin{equation*}
	x_{\alpha}\cdot x_{\beta} = (-1)^{\bra\alpha,\beta\ket}x_{\alpha+\beta},
	\end{equation*}
 and Poisson Lie bracket $[-,-]$ induced by the intersection form
 	\begin{equation*}
	[x_{\alpha}, x_{\beta}] = \bra\alpha,\beta\ket x_{\alpha}\cdot x_{\beta}. %(-1)^{\bra\alpha,\beta\ket}\bra\alpha,\beta\ket x_{\alpha+\beta}.
	\end{equation*}
A central charge $Z:\Gamma\to\C$ acts on $\C[\Gamma]$ as a derivation: $Z(x_{\alpha})= Z(\alpha)x_{\alpha}$.
	
The \emph{twisted torus} is 
	\begin{equation*}
	\torus := \big\{\xi:\Gamma\to\C^*:\xi(\alpha+\beta)=(-1)^{\bra\alpha,\beta\ket}\xi(\alpha)\xi(\beta) \big\}.
	\end{equation*}
Elements of $\C[\Gamma]$ act as characters on $\torus$:
	\begin{equation*}
	x_{\alpha}(\xi)=\xi(\alpha)\in\C^*,
	\end{equation*}
and $Z$ extends to the twisted torus $\torus$ via
	\begin{equation}\label{Z_deriv_and_torus}
	(Z\cdot \xi)(\alpha) = Z(\alpha)\xi(\alpha),
	\end{equation}
for every $\alpha\in\Gamma$, $\xi\in\torus$.

It is useful to introduce the maps% on $\torus$
	\begin{equation*}
	%\th:=\ln\xi:\Gamma\to\C.
	\th:=\ln\xi:\Gamma\to\R\times[0,2\pi[\cdot \ii,
	\end{equation*}
satisfying
	\begin{equation*}
	\th(\alpha+\beta)= \pi\ii\bra\alpha,\beta\ket + \th(\alpha)+\th(\beta) \mod 2\Z\pi\ii.
	\end{equation*}
Given any basis $\{\gamma_1,\dots,\gamma_n\}$ of $\Gamma$, a generic element $\xi$ of $\torus$ is determined by
	\begin{equation*}
	\xi_1:=\xi(\gamma_1),\dots, \xi_n:=\xi(\gamma_n),
	\end{equation*}
or by logarithmic coordinates
	\begin{equation*}
	\th_1:=\th(\gamma_1),\dots, \th_n:=\th(\gamma_n).
	\end{equation*}
We can also interpret $\theta_i$ as functions on the torus with non-trivial monodromy or make other choices of the branch of the complex logarithm: section \ref{sec_solutions} and \ref{sec_hamilt} would then require minor modifications.

\subsubsection*{Doubling construction}
A BPS structure $(\Pi,Z,\Omega)$ can be embedded into a richer structure, via \emph{doubling the construction}, \cite[{Sect.\,2.6}]{KS}. This is particularly useful when the intersection form $\bra-,-\ket$ is degenerate. To this end, the lattice $\Pi\oplus\Pi^{\vee}$, where $\Pi^{\vee}:=\Hom(\Pi,\Z)$, is considered. $\Pi\oplus\Pi^{\vee}$ is endowed with a non-degenerate skew-symmetric bilinear form denoted again by $\bra-,-\ket$ and defined as follows 
	\begin{align}
	&\text{for } \alpha',\alpha''\in\Pi,\ \nu',\nu''\in\Pi^{\vee} \notag \\
	&\bra (\alpha',\nu'),(\alpha'',\nu'')\ket = \bra\alpha',\alpha''\ket + \nu''(\alpha')-\nu'(\alpha'').\label{doubled_pairing}
	\end{align}
A doubled BPS structure is obtained by extending the central charge $Z$ and the BPS spectrum $\Omega$ to $\Pi\oplus\Pi^{\vee}$. We set 
	\begin{equation}\label{Z_Omega_doubled}
	Z(\beta,\nu):=Z(\beta),\quad \text{and}\quad \Omega(\beta,\nu)=\begin{cases} \Omega(\beta) &\text{if }\nu=0\\ 0&\text{otherwise}\end{cases}.
	\end{equation}
\begin{definition}\label{def_double} We refer to $\Pi\oplus\Pi^{\vee}$ as the doubled lattice and to this procedure the \lq\lq doubling procedure\rq\rq . With the choice \eqref{Z_Omega_doubled} above, $(\Pi\oplus\Pi^{\vee},Z, \Omega)$ is called a \emph{doubled BPS structure}. 
\end{definition} 
\begin{rmk}If $(\Pi,Z,\Omega)$ is an integral convergent uncoupled BPS structure, then so is its double $(\Pi\oplus\Pi^{\vee},Z,\Omega)$, and any $(\gamma,\bra-,\gamma\ket)\in\Pi\oplus\Pi^{\vee}$ is null.
\end{rmk}
A basis $\left\{\gamma_1,\dots,\gamma_m\right\}$ of $\Pi$ can be completed to a basis $\left\{\gamma_1,\dots,\gamma_m, \gdual_1,\dots,\gdual_m\right\}$ of $\Pi\oplus\Pi^{\vee}$, with  $\gdual_1,\dots, \gdual_m\in\Pi^{\vee}$ defined by
	\begin{equation}
	\gdual_j\left(\gamma_k\right) = \delta_{jk} \quad \forall j,k=1,\dots,m.
	\end{equation}
The twisted torus $\torus$ associated with a doubled BPS structure inherits logarithmic coordinates
	\begin{equation*}
	\th_j:=\th(\gamma_j),\quad \th_j^{\vee}:=\th(\gamma_j^{\vee}),\quad j=1,\dots, m,
	\end{equation*}
and comes equipped with the symplectic form $\omega=-\sum_{j=1}^{m} d\th_j \wedge d\th_j^{\vee}$.

\section{Riemann-Hilbert problems}\label{sec_RHpbs}
A Riemann-Hilbert (RH) problem classically consists in finding 
maps from $\C^*$ to a complex manifold with prescribed jumps across the supports of curves in $\C^*$. See for instance \cite[{Chapter\,3}]{fokas_painleve} for a brief introduction to the topic. Suppose we are given a complex manifold $\mathcal{M}$ together with a complex Lie group $\G$ acting on it, the union $\Sigma$ of supports of curves in $\C^*$ intersecting transversally at the origin, and a map $S:\Sigma\to \G$. Solving the RH problem defined by $S$ and with values in $\mathcal{M}$ means seeking a piecewise holomorphic function $\Psi:\C^*\setminus\Sigma\to \mathcal{M}$ such that for every $t\in\Sigma$ the limits $\Psi_{\pm}(t)$ of $\Psi$ from the opposite sides of $\Sigma$ exist and satisfy
	\begin{equation*}
	\Psi_+(t) = \Psi_-(t)S(t),
	\end{equation*}
and $\Psi$ has fixed constant limit $\lim_{t\to 0}\Psi(t)$ along any direction in $\C^*\setminus\Sigma$.

Existence of a solution is not guaranteed in general. The problem in the scalar case (i.e.\ when $\mathcal{M}=\C$) was widely treated for instance in \cite[{Muskhelishvili,\,1946}]{muske} or \cite[{Gakhov,\,1966}]{gakhov_boundary_pbs}, and solved for $S(t)$ H\"older continuous on the contour $\Sigma$ apart from a finite number of points. The solution to a scalar Riemann-Hilbert problem is unique provided that its restriction to a holomorphicity sector $\Delta$ can be continued to an invertible function on its closure $\bar{\holosect}\subset\C\PP^1$.

\subsection{RH problems for finite BPS structures}\label{sec_RHpb}

A ray-finite, integral, convergent BPS structure $(\Gamma,Z,\Omega)$ induces naturally a RH problem with values in the automorphism group $\Aut(\torus)$ of the twisted torus. Heuristically, attached to any active ray $\ell$ there is a transform $\stokes(\ell)$ defined by pull-back in $\C[\Gamma]$
	\begin{equation*}
	\stokes^*(\ell) : x_{\beta}\mapsto x_{\beta} \cdot \prod\limits_{Z(\alpha)\in\ell} (1-x_{\alpha})^{\Omega(\alpha)\bra\beta,\alpha\ket} .
	\end{equation*}
We refer to \cite{KS,bridg_RHpb} for the fundational issues about $\stokes(\ell)$ and the general definition. $\stokes(\ell)$ can be viewed as the time 1 Hamiltonian flow of a function on an open subset $U_{\ell}\subset\torus$ with respect to the Poisson bracket $\{-,-\}=[-,-]$ on $\torus$. \cite[{Proposition\,4.1}]{bridg_RHpb} states that there exist such an open subset $U_{\ell}$ where the power series 
	\begin{equation*}
	\dt(\ell):= \sum_{Z(\alpha)\in\ell}\Omega(\alpha)\sum_{h\geq 1} \frac{x_{\alpha}^h}{h^2} = \sum_{Z(\alpha)\in\ell} \Omega(\alpha)Li_2(x_{\alpha}),
	\end{equation*}
is absolutely convergent, and the time 1 Hamiltonian flow of this map is the holomorphic map $\stokes(\ell):U_{\ell}\to\torus$.

Let $\Sigma\subset\C^*$ be the union of active rays 
\begin{equation*}
\Sigma:=\bigcup_{\gamma\text{ active}}\ell_{\gamma}.
\end{equation*}
\begin{pb}\label{autTvaluedRHpb}
The $\Aut(\torus)$-valued RH problem attached to $(\Gamma,Z,\Omega)$ consists in finding  a piecewise holomorphic map $\Psi: \C^*\setminus\Sigma \to \Aut(\torus)$ with discontinuities on each component $\ell_{\gamma}$ of $\Sigma$ given by the composition with $\stokes(\ell_{\gamma})$ (jumping condition), and with  asymptotic behaviour near the origin (asymptotic condition)
\begin{equation*}
\lim_{t\to 0} \Psi(t)\circ \exp(Z/t) = \Id,
\end{equation*}
where the action of $Z$ on $\torus$ is given in \eqref{Z_deriv_and_torus}.
\end{pb}

For any fixed point $\xi\in\torus$, such an $\Aut(\torus)$-valued RH problem induces a problem with values in $\torus$ simply by evaluating on $\xi$ any automorphism of the torus.
\begin{definition}\label{TvaluedRHpb}
The $\torus$-valued RH problem for $(\Gamma,Z,\Omega)$ is defined as the problem of finding $\hat{X}:\C^*\setminus\Sigma\to\torus$, with discontinuities $\stokes(\ell_{\gamma})\in\Aut(\torus)$ and  asymptotic behaviour $\lim_{t\to 0}\hat{X}(t)\cdot\exp(Z/t) = \xi$.
\end{definition}
Notice that the hypothesis of ray-finiteness of the structure $(\Gamma,Z,\Omega)$ is essential to define the problem in Definition \ref{TvaluedRHpb}, while BPS structures might present countably many active rays.

\subsection{Scalar RH problems for uncoupled BPS structures}\label{scalarRHpb}

If moreover $(\Gamma,Z,\Omega)$ is generic and uncoupled, then for any choices of $\xi$ the problem of Definition \ref{TvaluedRHpb} can be turned into a scalar problem (\cite[{sections\,4.2}]{bridg_RHpb}) involving maps 
	\begin{equation}\label{defYbeta}
	\hat{Y}_{\beta}(t):=x_{\beta}(\hat{X}(t)\cdot e^{Z/t} \cdot \xi^{-1}):\C^*\setminus\Sigma\to\C,
	\end{equation}
and defined by functions $\stokes_{\ell}:\Gamma\times\C^*\to\C$,
	\begin{equation*}
	\stokes_{\ell}(\beta,t):=\prod\limits_{\gamma\in\ell}(1-\xi(\gamma)e^{-Z(\gamma)/t})^{\Omega(\gamma)\bra\beta,\gamma\ket}.
	\end{equation*}
\begin{rmk} This does not applies to non-uncoupled BPS structures. It depends on the fact that $\stokes_{\ell}(\beta,t)$ is trivial when $\beta$ is active. %Uncoupledness also implies commutativity of $\stokes_ell_1$, $\stokes_ell_2$ for any $\ell_1,\ell_2$.
\end{rmk}
\begin{definition}\label{uncoupled_stokes}
In analogy with the theory of differential equations, we call $\stokes_{\ell}(\beta,t)$ a \emph{Stokes factor} of the problem. For any fixed $\beta\in\Gamma$, we will call also $\stokes_{\ell}(t)$ a Stokes factor.
\end{definition}

Let $(\Gamma,Z,\Omega)$ be an integral generic convergent \emph{uncoupled} BPS structure and fix $\xi\in\torus$. For any ray $l$, let $\H_{l}$ be the open half-plane centred in $l$
	\begin{equation*}
	\H_l=\{v\cdot z\in\C^* : v\in l, -\pi/2<\arg(z) < \pi/2\}.
	\end{equation*}
\begin{pb}[{\cite[Problem\,4.3]{bridg_RHpb}}]\label{holoRHpb} 
For each non-active ray $r\in\C^*$ and for every $\beta\in\Gamma$, we seek a holomorphic function $Y_{\beta,r}:\H_r\to\C^*$ such that the following conditions are satisfied.
\begin{itemize}
\item[$RH_1$] Suppose that two generic rays  $r_1\neq r_2$  form the boundary rays of a convex sector $\Delta\subset\C^*$ taken in clockwise order, then for all $t\in\H_{r_1}\cap\H_{r_2}$ with $0<|t|\ll 1$, 
	\begin{equation*}
	Y_{\beta,r_1}(t)= Y_{\beta,r_2}(t) \cdot \prod_{Z(\gamma)\in\Delta}(1-\xi(\gamma)e^{-Z(\gamma)/t})^{\Omega(\gamma)\bra\beta,\gamma\ket};
	\end{equation*}
\item[$RH_2$]\label{hRH_cond_lim0} $\lim\limits_{\substack{t\to 0 \\ t\in\H_r}} Y_{\beta,r}(t)=1$;
\item[$RH_3$] there exist $k=k(\beta,r)$ such that for all $t\in\H_r$ with $|t|\gg 0$, 
	\begin{equation*}
	|t|^{-k} < | Y_{\beta,r}(t) | <|t|^k. 
	\end{equation*}
\end{itemize}
\end{pb}

Problem \ref{holoRHpb} has the advantage of involving complex functions, moreover it admits at most one solution, \cite{bridg_RHpb}. A solution $\{Y_{\beta,r}\}_{\beta,r}$ of \ref{holoRHpb} is related with the functions $\hat Y_{\beta}$ in \eqref{defYbeta} via analytic continuation to half-planes of the restriction of $\hat Y_{\beta}$ to the holomorphicity sectors. 

Problem \ref{holoRHpb} was solved in this formulation in \cite{bridg_RHpb} for the special fixed point $\xi\equiv 1$, but it is easily seen that it often does not admit solution. 

\begin{prop}\label{prop_holoRH_nosol} Suppose that there exists an active class $\gamma$ such that $\th(\gamma)=\ln(\xi(\gamma))\neq 0$ and $Z(\gamma)/\th(\gamma)\in\overline\H_{\ell_{\gamma}}$. Then the Problem \ref{holoRHpb} does not admit a solution.
\end{prop}
\begin{proof}
In the hypothesis of the Proposition, suppose for instance that $\Z(\gamma)/\th(\gamma)$ is in the convex sector between $\ii\ell_{\gamma}$ and $\ell_{\gamma}$. Suppose moreover that the Problem \ref{holoRHpb} admit solutions $Y_{\beta,r}:\H_r\to\C^*$ for every non active ray $r$. Let $\beta$ such that $\bra\beta,\gamma\ket>0$. There are two distinct non-active rays $r_1$ and $r_2$ such that $\Delta=\H_{r_1}\cap \H_{r_2}$ contains $t = Z(\gamma)/\theta(\gamma)$ and $\ell_{\gamma}$. This implies that $\left(1-e^{\th(\gamma)-Z(\gamma)/t}\right)$ divides $Y_{\beta,r_1}\cdot \left(Y_{\beta,r_2}\right)^{-1}$ for every $\beta\in\Gamma$, with $Y_{\beta,r_i}$ never vanishing in $\Delta$. But $t=Z(\gamma)/\th(\gamma)\in\Delta$ is a zero of $\left(1-e^{\th(\gamma)-Z(\gamma)/t}\right)$, yielding a contradiction.
\end{proof}
\noindent In particular, if $\theta(\gamma)\in\R\setminus\{0\}$, then $Z(\gamma)/\theta(\gamma)$ lies in one of the active rays $\pm\ell_{\gamma}$. 

Proposition \ref{prop_holoRH_nosol} is not a counterexample to the existence of piecewise continuous solutions to \ref{TvaluedRHpb} and we reformulate the scalar Riemann-Hilbert problem in terms of meromorphic functions.

\begin{pb}[Meromorphic RH problem]\label{meroRHpb}
For every $\beta\in\Gamma$ and for each non-active ray $r$, we seek a meromorphic function $Y_{\beta,r}: \H_{r}\to \C\PP^1$ satisfying the following conditions:
\begin{enumerate}\setcounter{enumi}{-1}
\item[$RH_0$]\label{mRH_cond_mero} $Y_{\beta,r}$ is holomorphic and $\C^*$-valued away from a finite number of zeroes or poles in position $t=\frac{Z(\gamma)}{\th(\gamma)+2k\pi\ii}$, $\gamma\in\Gamma$, for some $k\in\Z$;
\item[$RH_1$] suppose that two generic rays  $r_1\neq r_2$  form the boundary rays of a convex sector $\Delta\subset\C^*$ taken in clockwise order, then  
	\begin{equation*}
	Y_{\beta,r_1}(t)= Y_{\beta,r_2}(t) \cdot \prod_{Z(\gamma)\in\Delta}(1-\xi(\gamma)e^{-Z(\gamma)/t})^{\Omega(\gamma)\bra\beta,\gamma\ket} \quad \forall\ t\in\H_{r_1}\cap\H_{r_2};
	\end{equation*}
\end{enumerate}
and $RH_2$, $RH_3$ hold as in \ref{holoRHpb}, away from some $t=\frac{Z(\gamma)}{\th(\gamma)+2k\pi\ii}$, $\gamma\in\Gamma$, $k\in\Z$.
\end{pb}
Notice that we keep the same notation for conditions in \ref{meroRHpb} as in  \ref{holoRHpb}, although the domain is different. 

\begin{prop}\label{prop_mero_uniq}
The solution to the Problem \ref{meroRHpb} associated with a finite BPS structure, when exists, is unique if and only it the order of zero/pole at any critical point is specified.
\end{prop}
\begin{proof}
Fix a vanishing order of a finite number of points. The proof goes as in \cite[{Lemma 4.9}]{bridg_RHpb}, with minor modifications. The argument is a standard application of the Liouville theorem, see also \cite[{Chapter\,3}]{fokas_painleve} as an example. 
\end{proof}

\begin{definition} We say that a solution is minimal if its finitely many critical points (zeroes or poles) associated with any $\gamma$ (that is in position $\frac{Z(\gamma)}{\theta(\gamma)+2k\pi i}$, $k\in\Z$) are simple and lie on the same side of $\ell_{\gamma}$. \end{definition}
\noindent By Proposition \ref{prop_mero_uniq}, there exists at most one \lq\lq minimal\rq\rq\ solution to Problem \ref{meroRHpb}.

\begin{rmk}[{\cite[{Remark 4.10}]{bridg_RHpb}}] 
When $(\Gamma, \Omega, Z)$ is an integral convergent uncoupled BPS structure with trivial pairing $\bra-,-\ket$, it is convenient doubling it. The RH problem for an uncoupled convergent BPS structure and its double are equivalent. A vector $(\gamma, \bra-,\gamma\ket)\in\Gamma\oplus\Gamma^\vee$ is null. The formulation of the meromorphic problem \ref{meroRHpb} implies that $Y_{\beta,r}\equiv 1$ for null vectors $\beta$. Therefore $Y_{(\gamma, \bra-,\gamma\ket),r}\equiv 1$. Recall that the functions $Y_{(\gamma,\nu),r}$ are modelled on \eqref{defYbeta}, composing a map on the torus with a character $x_{(\gamma,\nu)} = (-1)^{\bra(\gamma,0),(0,\nu)\ket}x_{(\gamma,0)}\cdot x_{(0,\nu)}$ and are thus multiplicative in the index. This implies that $Y_{\gamma,r}\equiv Y_{\bra-,\gamma\ket,r}$  (note the abuse of notation). At the same time, the doubled solution carries more information.
\end{rmk}

\section{Complex analysis}\label{sec_complex_analysis}

In this section we introduce the complex multivalued function
	\begin{equation}\label{defSol}
	\mmg{x}{y}:=\frac{\Gamma(x+y)\cdot e^y}{y^{x+y-\frac{1}{2}}\cdot\sqrt{2\pi}}
	\end{equation}
defined for $x\in\C$ and $y\in\C^*$ as a modification of the Gamma function. We study a number of properties (Lemma \ref{lem_Sol_jump}) and we provide an integral expression for $\mmg{x}{y}$ (Lemma \ref{lem_jumpXX}) that justifies why it will define a basis of solutions to \ref{meroRHpb}. To this end we also introduce the function $S:\C\times\C\to\C$, symmetric with respect the exchange of the two variables,
\begin{equation}\label{S}
S(x,y)=S(y,x)=1-e^{x}e^{y}.
\end{equation}
Multivaluedness of \eqref{defSol} depeds on $\exp\left( \left(x+y-\frac{1}{2}\right)\ln y -\ln y\right)$. For a chosen branch of the logarithm, it is a meromorphic function in two variable with poles prescribed by $\Gamma(x+y)$ at points $(x,x+2\pi k\ii)$, $k\in\Z_{\leq 0}$. Later, we will consider $$\mg{x}{y}:=\mmg{x}{y},$$ regarded as a family of meromorphic functions defined for $y\in\C^*\setminus\Rp$, parametrised by a choice of $x=\ln e^x$, $k\pi\leq\Im(x)<(k+1)\pi$, and with aligned poles at $y=x+2\pi k \ii$, $k$ a negative integer.

\begin{lem}\label{lem_Sol_jump} \begin{enumerate}
\item Assume $\Im(x)>0$, then \begin{equation*}\mmg{x}{-y}\cdot \mmg{1-x}{y} = \begin{cases}
S(-2\pi\ii x,2\pi\ii y)^{-1} &\text{ if }\Im(y)>0 \\
S(2\pi\ii x,-2\pi\ii y)^{-1} &\text{ if }\Im(y)<0
\end{cases},\end{equation*}
\item $\mmg{1+x}{y}=\left(1+\frac{x}{y}\right)\mmg{x}{y}$,
\item $\Lambda_{x}^{-1}(y)$ has algebraic behaviour around the origin, i.e.\ there is $m\in \N$ such that $|y|^{m}<|\Lambda_{x}^{-1}(y)|<|y|^{-m}$, as $|y|\to 0$,
\item $\lim_{|y|\to\infty}\mg{x}{y}=1$.
	\end{enumerate}
\end{lem}
\begin{proof}
For \emph{1.}\ recall the Euler reflection formula for the Gamma function
	\begin{equation*}
	\Gamma(z)\Gamma(1-z) = \frac{\pi}{\sin(\pi z)} = \frac{2\pi\ii}{e^{\ii\pi z}-e^{-\ii\pi z}}.
	\end{equation*}
We write $y^{1-x+y-\frac{1}{2}}$ occurring in $\mmg{1-x}{y}$ as $$\exp\left( \left((-x+y+\frac{1}{2}\right)\ln y \right),$$ and $(-y)^{x-y}$ occurring in $\mmg{x}{-y}$ as 
	\begin{equation*}
	\exp\left( \left(x-y-\frac{1}{2}\right)\ln y \mp \pi\ii(x-y-\frac{1}{2}\right) \quad \text{ if } \pm\Im(y)>0.
	\end{equation*}
Then, if for instance $\Im(y)>0$, $\mmg{x}{y}\cdot \mmg{x}{-y}$ equals
	\begin{equation*}
	2\pi e^{\pi\ii/2} e^{-\pi\ii (x-y)} \frac{e^{\ii\pi (x-y)}-e^{-\ii\pi (x-y)}}{2\pi\ii}= 1-e^{-2\pi\ii(x-y)}.
	\end{equation*}
For \emph{2.} use the property $\Gamma(z)=z\Gamma(1+z)$.\\
\emph{3.}\ is clear as, for $|y|<\epsilon$, $\Lambda_{x}^{-1}(y)$ is bounded by a function that goes as a holomorphic function times $y^\epsilon$.\\
\emph{4.}\ follows from the following formula for the logarithm of the shifted Gamma function \cite[Chapter\,1.1]{magnus} valid for any $N\in\N$, $z,a\in\C$,
	\begin{align*}
	\ln\Gamma(z+a) = \Big(z+a-&\frac{1}{2}\Big) \log z - z + \ln(\sqrt{2\pi}) + \\
	&+\sum_{m=1}^N \frac{(-1)^{m+1} \cdot B_{m+1}(a)}{m(m+1)} \cdot z^{-m} + O(z^{-N-1/2}).
	\end{align*}
\end{proof}
\noindent One also deduces that the logarithm $\ln\mmg{x}{y}$ has formal asymptotic expansion
	\begin{equation*}
	\ln \mmg{x}{y} \sim \sum_{m=1}^{+\infty} \frac{(-1)^{m+1}B_{m+1}(x)}{m(m+1)} \cdot y^{-m},
	\end{equation*}
as $y$ lies in any convex sector of $\C^*$ not containing $\R_{<0}$.
\begin{rmk}
$\mmg{1}{y}=\mmg{0}{y}$ coincides with the modified Gamma function $\Lambda$ introduced in \cite[{Section 3.2}]{bridg_RHpb}. Compare also with \cite[{Eq.\ 3.10}]{gaiotto}.
\end{rmk}
	
$\mmg{x}{y}$ can be defined as the analytic continuation of an integral expression. For $(\th,w)\in\C^2$, $0<\Im(\th)<2\pi$ and $\Im(w)\neq 0$, we consider the function 
	\begin{equation*}
	\XXp(\th,w) := \exp \left( -\frap \int_{\Rp} \frac{\ln(1-s/w)}{e^{-\th+s}-1} d s - \frap \int_{-\Rp} \frac{\ln(1-s/w)}{e^{\th-s}-1} d s \right).
	\end{equation*}
It can be extended over $\Im(\th)>2\pi$ by
	\begin{equation*}
	\XXp(\th,w) = \XXp(\th-2\pi\ii,w) \cdot \left(1-\frac{\th}{w}\right).
	\end{equation*}
Lemmas \ref{lemTom} and \ref{lem_grad} are aimed to prove the following Theorem.
\begin{thm}\label{thmSolInt}
When $\Im(\th)>0$, $\Im(w)<0$, 
	\[\mmg{\frac{\th}{2\pi\ii}}{-\frac{w}{2\pi\ii}}=\XXp(\th-2\pi\ii,w).\] 
\end{thm}
The shift of $2\pi\ii$ in the Theorem above is essentially related with different choices of the branch of the logarithm.
\begin{lem}\label{lemTom}  When $\Im(w)<0$,
	\begin{equation*}
	\XXp(0,w) = \mmg{0}{-\frac{w}{2\pi\ii}}.
	\end{equation*}
\end{lem}
\begin{proof}
To verify the equation we use Binet's second formula \cite[{section\,12.32}]{whittaker}, valid for $\Re(\ii w)>0$, expressing  
	\begin{equation*}
	\ln\Gamma\Big(\frac{\ii w}{2\pi} \Big) + (\ii w) - \frac{1}{2} \ln(2\pi) - \left(\ii w-\frac{1}{2}\right) \ln (\ii w)
	\end{equation*}
as
	\begin{equation*}
	2\int_{0}^{+\infty} \arctan(2\pi s/(\ii w)) \frac{1}{e^{2\pi s}-1} ds.
	\end{equation*}
Since $\arctan(z)= -\frac{1}{2\ii} \left[ \ln(1-\ii z) - \ln(1+\ii z)\right]$ in $\C^*\setminus \big(\pm[\ii,\ii\infty[\big)$, the thesis follows.
\end{proof}
\begin{lem}\label{lem_grad} 
In the domain of definition of $\ln\XXp(\th,w)$, 
	\[\left(\frac{\del}{\del\th}+\frac{\del}{\del w}\right) \ln\XXp(\th,w)  = -\frac{1}{2\pi\ii}\left(\th+\pi\ii\right) w^{-1}.\]
\end{lem}
\begin{proof}
Integrating by parts we have
	\[\ln\XXp(\theta,w)= -\frac{1}{2\pi i} \int_{\R_{>0}} \frac{\ln(1-e^{\theta-s})}{w-s} ds + \frac{1}{2\pi i} \int_{-\R_{>0}} \frac{\ln(1-e^{-\theta+s})}{w-s} ds
	\]
and hence
	\begin{equation*}
	 \deld{\th}\ln\XXp=-\frap\int_0^{\infty} \frac{ w^{-1} d s}{(1-e^{s-\th})(1-sw^{-1})} -\frap \int_0^{-\infty} \frac{w^{-1}d s}{(1-e^{\th-s})(1-sw^{-1})}.
	\end{equation*}
On the other hand $\deld{w}\ln\XXp=-w^{-2}\deld{w^{-1}}\ln\XXp$ and
	\begin{equation*}
	-w^{-2}\deld{w^{-1}}\ln\XXp=-\frap\int_0^{+\infty} \frac{-sw^{-2} d s}{(1-e^{s-\th})(1-sw^{-1})} - \frap\int_0^{-\infty} \frac{-sw^{-2} d s}{(1-e^{\th-s})(1-sw^{-1})}.
	\end{equation*}
Therefore
	\begin{equation*}
	 \left(\frac{\del}{\del\th}+\frac{\del}{\del w}\right) \ln \XXp = -\frac{w^{-1}}{2\pi\ii} \left( \ln(1-e^{\th}) - \ln(1-e^{-\th}) \right) = -\frac{w^{-1}}{2\pi\ii} \left(\th+\pi\ii\right),
	\end{equation*}
as $\Im(\th)>0$.
\end{proof}
\begin{proof}[{Proof of Theorem \ref{thmSolInt}}]
By Lemma \ref{lem_grad}, we can transport $\XXp(0,w-\th)$ along diagonal directions. Let $0<\Im(\th)<2\pi$ and $\Im(w)<0$.
	\begin{equation*}
	\ln\XXp(\th,w)= \ln\XXp(0,w-\th) + \int_{\eta}d\ln\XX,
	\end{equation*}
where $\eta$ is the path $\eta(x)=(x,w-\th+x)$, $x\in[0,\th]$. 
	\begin{equation*}
	\int_{\eta}d\ln\XXp= -\int_{-\th+w}^w \frap\left(x'-w+\th+\pi\ii\right)(x')^{-1} dx'
= -\frap \left(\th + (\th-w+\pi\ii)\ln \left(\frac{w}{w-\th}\right)\right).
\end{equation*}
Then
	\begin{equation}\label{LambdaPi}
	\XXp(\th,w) = \Lambda\left(0,-\frac{w-\theta}{2\pi\ii}\right)^{-1}\cdot e^{-\frac{\th}{2\pi\ii}} \cdot\left(\frac{w}{w-\th}\right)^{\frac{\th-w}{2\pi\ii}-\frac{1}{2}}.
	\end{equation}
We write $\left(\frac{w}{w-\th}\right)^{-\frac{1}{2}}$ as $\left(\frac{w}{w-\th}\right)^{\frac{1}{2}}\left(1-\frac{\th}{w}\right)$. Manipulating the equation \eqref{LambdaPi}, one obtains 
	\begin{equation*}
	\XXp(\th,w) =  \frac{\Gamma\left(-\frac{w-\th}{2\pi\ii}\right)\cdot e^{-\frac{w}{2\pi\ii}}}{\sqrt{2\pi} \cdot \left(-\frac{w}{2\pi\ii}\right)^{-\frac{w-\th}{2\pi\ii}-\frac{1}{2}}}\left(1-\frac{\th}{w}\right).
	\end{equation*}
The thesis follows from \emph{2.}\ in Lemma \ref{lem_Sol_jump}.
\end{proof}
The function $\XXp$ has the form of a classical solution to a Riemann-Hilbert problem. Such an integral expression is the basis solution for an analogous Riemann-Hilbert problem considered in \cite{FGS}, section \ref{sec_FGS}. We can look at $\XXp$ as a piecewise function in $0<|\Im(\th)|<2\pi$, $\Im(w)\neq 0$, satisfying the symmetry
	\begin{equation}\label{eq_symmXXp}
	\XXp(\th,w) = \XXp(-\th,-w)^{-1},
	\end{equation}
and with discontinuities prescribed by $S$. This can be shown via a direct integral contour argument.
\begin{lem}\label{lem_jumpXX} Denote by $\XXp(\th,w_0^{\pm})$ the limits of $\XXp(\th,w)$ as $w$ approaches a point $w_0$ clock-wise and counter-clock-wise respectively. Then if $w_0\in\Rp$,
	\begin{equation*}
	\XXp(\th,w_0^-) = \XXp(\th,w_0^+) \cdot S(\th,-w_0),
	\end{equation*}
while if $v_0\in\Rm$ 
	\begin{equation*}
	\XXp(\th,v_0^-) = \XXp(\th,v_0^+) \cdot S(-\th,v_0)^{-1}.
	\end{equation*}
\end{lem}
\begin{proof} Assume $w_0\in\Rp$.  Computing $\XXp(\th,w_0^{+})$ is equivalent to slightly deform the integral path $\Rp$ clock-wise in the lower half-plane and evaluate the function in $w_0$. We define the contour $\CC$: 
$$\CC = \CC_+ \cup \CC_{\delta} \cup \CC_-,$$ where $\CC_+=[w_0,\infty]\times\ii\e$, with inverse orientation, $\CC_{\delta}$ is a half-cycle of radius $\delta>0$ centred on $w_0$, and $\CC_-=[w_0,\infty]\times(-\ii\e)$ with standard orientation, $\e>0$. Then
\begin{equation}\label{contourint}
 \ln \XXp(\th,w_0^-) - \ln \XXp(\th,w_0^+)= \frac{1}{2\pi\ii} \int_{\CC} \frac{\ln\left(1-\left(s/ w\right)\right)}{e^{-\th+s} -1} d s,
\end{equation}
at the limit for $\e,\delta\to 0$. In \eqref{contourint}, the contribution from $\CC_{\delta}$ vanishes as $\delta\to 0$. $1-\frac{s}{w_0}\in\Rm$ if $s$ is real and has positive (resp.\ negative) imaginary part if $s\in\CC_+$ (resp.\ $\CC_-$). The contributions from $\CC_-$ and $\CC_+$ differ by $\frac{1}{2\pi\ii}\int_{w_0}^{\infty} \frac{-2\pi\ii}{e^{-\th} e^{ s}-1} d s$, that is $\ln(1-e^{\th-w_0})$, from which the thesis follows. The statement for $v_0\in\Rm$ can be deduced using \eqref{eq_symmXXp}.
\end{proof}

\section{Solution to the RH problem for uncoupled structures}\label{sec_solutions}

\subsection{Doubled $A_1$ BPS structure}\label{sec_explsol}
We first consider the simplest case, that is the doubled $A_1$ BPS structure with two active classes $\pm\alpha$.

\begin{exm}\label{exm_A1}
The doubled $A_1$ BPS structure is the datum of
\begin{itemize}
\item a lattice $\Gamma=\Z\cdot \alpha \oplus \Z\cdot \alpha^{\vee}$, endowed with the skew-symmetric non-degenerate pairing \eqref{doubled_pairing} and such that $\bra\alpha,\alpha^{\vee}\ket = 1$;
\item a group homomorphism $Z:\Gamma\to\C$ with $Z(\alpha^{\vee})=0$;
\item a map $\Omega:\Gamma\to\Z$ with 
	\begin{equation*}
	\Omega(\beta)=\begin{cases} 1 \text{ if }\beta=\pm(\alpha,0) \\
	0 \text{ otherwise}
	\end{cases}.
	\end{equation*}
\end{itemize}
\end{exm}
Fix $\xi=e^{\th}\in\torus$. We define $\ell:=Z(\alpha)\cdot\Rp$ one of the two opposite active rays, $\th:=\th(\alpha)$, $z:=Z(\alpha)$. The Stokes factors $\stokes(\beta,\pm\ell)=\left(1-\xi(\pm\alpha)e^{Z(\pm\alpha)/t}\right)^{\bra\beta,\alpha\ket}$ are trivial when $\beta\in\Z\cdot\alpha$. This implies that $Y_{\alpha,r}\equiv 1$. The solutions $\{Y_{\alpha^{\vee},r}\}_r$ to the meromorphic Riemann-Hilbert problem glue together to two meromorphic functions satisfying the following problem, \cite{bridg_RHpb}.
\begin{pb}\label{RHA1} Find two meromorphic functions
	\begin{equation*}
	Y_{\pm}:\C^*\setminus\pm\ii\ell\to\C\PP^1
	\end{equation*}
such that
\begin{enumerate}\setcounter{enumi}{-1}
\item $Y_{\pm}$ are holomorphic and nonvanishing away from zeroes/poles at $\frac{z}{\th(\alpha)\pm 2k\pi\ii}$, $k\in\N$,
\item $Y_{\pm}$ satisfy
	\begin{equation*}
	Y_+(t) = \begin{cases}
	 Y_-(t) \cdot (1-\xi(\alpha) e^{-z/t}) \quad \text{for } t\in\H_{\ell} \\
	Y_-(t) \cdot (1-\xi(-\alpha) e^{z/t}) \quad \text{for } t\in\H_{-\ell}
	\end{cases},
	\end{equation*}
\item $\lim_{t\to 0}Y_{\pm}(t) = 1$,
\item $Y_{\pm}(t)$ has at most algebraic growth when $|t|\gg 0$.
\end{enumerate}
\end{pb}
\begin{prop}\label{prop_xpm_A1}  
	\begin{equation*}
	Y_-(t):=\Lambda_{\frac{\th}{2\pi i}}\left(-\frac{z}{2\pi i t}\right) \quad \text{and} \quad Y_+(t):=\Lambda_{1-\frac{\th}{2\pi i}}^{-1}\left(\frac{z}{2\pi i t}\right)
	\end{equation*}
solve the Problem \ref{RHA1}.
\end{prop}
\begin{proof}
Observe that $t\in\ell$ if and only if $z/t\in\Rp$ and $t\in\H_{\ell}$ if and only if $\Im\left(\frac{z}{2\pi\ii t}\right)<0$. The Theorem follows from Theorems \ref{lem_Sol_jump}, \ref{thmSolInt} and Lemma \ref{lem_jumpXX}, as $\xi(-\alpha)=\xi(\alpha)^{-1}$. 
\end{proof}
$Y_+$ and $Y_-$ are respectively a holomorphic function with simple zeroes at points $\frac{z}{\th+2m\pi\ii}$, $m\in\N$, and a meromorphic function with simple poles at $\frac{z}{\th-2m\pi\ii}$, $m\in\N\setminus\{0\}$. Points $\frac{z}{\th+ 2k\pi\ii}$, $k\in\Z$, lie in a circle divided in two halves by $\ell$ or $-\ell$, and cluster at the origin. The circle degenerates to the origin if $\theta=0$. Every half-plane $\H_{\pm r}$, $r\neq\pm\ell$, contains then at most a finite number of those points. 
\begin{center}
\begin{tikzpicture}[rotate=20]
\filldraw[white,fill=red!10] (0,0) -- (1.4,0) arc (0:120:1.4) -- (0,0);
\filldraw[white,fill=red!10] (0,0) -- (1.4,0) arc (0:-60:1.4) -- (0,0);
\draw [thick] (0:0)--(0:2) node[above, right] {$\ell$};
\draw [thick] (0:0)--(0:-2) node[below, left] {$-\ell$};
\draw (90:0)--(90:1.7);
\draw [dashed] (90:0)--(90:-1.7);
\draw (180:-1.5)--(180:2);
\draw [thick,red] (30:0)--(30:1.7) node[above] {$r$};
\draw [ultra thin, dashed, domain=20:180] plot ({0.6+0.6*cos(\x)}, {0.6*sin(\x)});
\draw [red] (120:-1.7)--(120:1.7);
\draw [thick, dotted, domain=180:230] plot ({0.6+0.6*cos(\x)}, {0.6*sin(\x)});
\draw [red] ({0.6+0.6*cos(-30)}, {0.6*sin(-30)}) circle [radius=0.04];
\draw [red] ({0.6+0.6*cos(-55)}, {0.6*sin(-55)}) circle [radius=0.03];
\draw [red] ({0.6+0.6*cos(-75)}, {0.6*sin(-75)}) circle [radius=0.03];
\draw ({0.6+0.6*cos(235)}, {0.6*sin(235)}) circle [radius=0.015];
\draw [red] ({0.6+0.6*cos(245)}, {0.6*sin(245)}) circle [radius=0.02];
\draw [red] ({0.6+0.6*cos(255)}, {0.6*sin(255)}) circle [radius=0.025];
\draw [red] ({0.6+0.6*cos(268)}, {0.6*sin(268)}) circle [radius=0.025];
\end{tikzpicture}
\quad\quad\quad
\begin{tikzpicture}[rotate=20]
\filldraw[white,fill=red!10] (0,0) -- (1.4,0) arc (0:60:1.4) -- (0,0);
\filldraw[white,fill=red!10] (0,0) -- (1.4,0) arc (0:-120:1.4) -- (0,0);
\draw [ultra thin, dashed, domain=180:340] plot ({0.6+0.6*cos(\x)}, {0.6*sin(\x)});
\draw [thick] (0:0)--(0:2) node[above, right] {$\ell$};
\draw [thick] (0:0)--(0:-2) node[below, left] {$-\ell$};
\draw [dashed] (90:0)--(90:1.7);
\draw (90:0)--(90:-1.7);
\draw (180:-1.5)--(180:2);
\draw [thick,red] (-30:0)--(-30:1.7) node[above] {$r$};
\draw [red] (60:-1.7)--(60:1.7);
\draw [thick, dotted, domain=130:180] plot ({0.6+0.6*cos(\x)}, {0.6*sin(\x)});
\draw [red] ({0.6+0.6*cos(30)}, {0.6*sin(30)}) node[cross=1.5pt, red] {};
\draw [red] ({0.6+0.6*cos(55)}, {0.6*sin(55)}) node[cross=1.4pt, red] {};
\draw [red] ({0.6+0.6*cos(75)}, {0.6*sin(75)}) node[cross=1.4pt, red] {};
\draw [red] ({0.6+0.6*cos(92)}, {0.6*sin(92)}) node[cross=1.2pt, red] {};
\draw [red] ({0.6+0.6*cos(105)}, {0.6*sin(105)}) node[cross=1.2pt, red] {};
\draw [red] ({0.6+0.6*cos(115)}, {0.6*sin(115)}) node[cross=1.1pt, red] {};
\draw ({0.6+0.6*cos(125)}, {0.6*sin(125)}) circle node[cross] {};
\end{tikzpicture}
\\
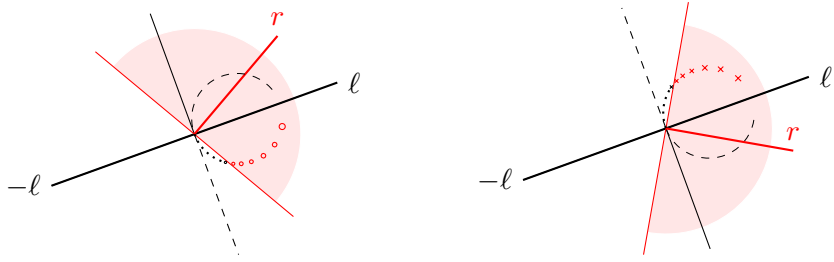
Example: poles of $Y_-(t)$ (left) and zeroes of $Y_+(t)$ (right) lying in a half-plane $\H_r$.\end{center}
Call $\{Y_{\beta,r}\}$ a system of solutions to \ref{meroRHpb}. For null-vectors $\beta$, $Y_{\beta,r}\equiv 1$. For every non-active ray $r$ occurring clock-wise between $\ell$ and $-\ell$, take
 \begin{equation*}
Y_{\gamma^{\vee},r}(t):=Y_{+\,|\H_{r}}(t),\quad Y_{\gamma^{\vee},-r}(t):=Y_{-\,|\H_{-r}}(t).
\end{equation*}
By Proposition \ref{prop_xpm_A1}, $Y_{\gamma^{\vee},\pm r}(t)$ satisfy $RH_1$, $RH_2$, $RH_3$, and provide a solution to Problem \ref{meroRHpb} with only simple zeroes/poles. The shift $\th\mapsto \th+2k\pi\ii$, $k\in\Z$, produces another solution with shifted simple zeroes/poles. Notice moreover that, since the jumping factors $\stokes(\pm\ell)(t)=S(\pm\theta, \mp z/t)$, defined in \eqref{S}, admits a factorisation in an infinite product $\prod S_j(t)$, we have that, if $\{Y_r\}$ satisfies $RH_0$--$RH_3$, so $\{Y_r\cdot S_j\}$ does. 

\subsection{General case}\label{sec_gensol}

Solutions to the Riemann-Hilbert problems in the finite uncoupled case are obtained by superimposing the solution in the doubled $A_1$ case along any \lq\lq active direction\rq\rq . Let $(\Gamma,Z,\Omega)$ denote an integral uncoupled convergent BPS structure. For any generic ray $r$, define 
	\begin{equation}\label{Gactiver}
	\Gamma_r^{\Omega} := \left\{ \gamma\in\Gamma : \Omega(\gamma)\neq 0 \text{ and } 0<\arg \left(\frac{v}{Z(\gamma)}\right)<\pi,\ \forall v\in r \right\},
	\end{equation}
the set of active classes $\gamma$ whose corresponding active ray lies \lq\lq on the right\rq\rq\ of $r$. We also define 
	\[
	\epsilon_{(\beta,\gamma)}=\sgn\bra\beta,\gamma\ket \quad \text{and} \quad a_{(\beta,\gamma)}=\begin{cases}\theta(\gamma) & \text{if } \epsilon_{\beta,\gamma}=1 \\
	1-\theta(\gamma) & \text{if } \epsilon_{\beta,\gamma}=-1 \end{cases}.
	\]
For $\beta\in\Gamma$, a minimal solution $Y_{\beta,r}$ is given by the restriction to $\H_r$ of 
	\begin{equation}\label{sol_finite_uncoupled}
	\prod_{\gamma\in\Gactiver} {\mg{a_{(\beta,\gamma)}}{-\epsilon_{(\beta,\gamma)}\frac{Z(\gamma)}{2\pi\ii t}}}^{\Omega(\gamma)\bra\beta, \gamma\ket}.
	\end{equation}
The next Theorem \ref{thm_sol_meroRH_finite} then follows.
\begin{thm}\label{thm_sol_meroRH_finite}
The meromorphic Riemann-Hilbert problem \ref{meroRHpb} for finite uncoupled BPS structures admits solutions for every choice of $\xi\in\torus$. The unique minimal solution is given by \eqref{sol_finite_uncoupled}.
\end{thm}

\subsection{Relation with the work \cite{FGS}}\label{sec_FGS}
In \cite{FGS} a Riemann-Hilbert problem strictly related to those considered here is studied. They are conceptually different, but the solutions in the uncoupled case are formally the same. 
In this section we briefly describe the relation between their approach and ours.   
 
The problem in \cite{FGS} is stated for \lq\lq positive BPS structures\rq\rq . It is viewed as the \lq\lq  conformal limit\rq\rq\ of a Riemann-Hilbert problem stated in \cite{GMN}, which has a different asymptotic behaviour at infinity. By \lq\lq positive BPS structure\rq\rq\ we mean a triple $(\Gamma,Z,\Omega)$ satisfying Definition \ref{def_BPS} apart from (i), together with a choice of a convex cone $\Gamma^+\subset\Gamma$ consisting in non-negative linear combinations of elements of a fixed basis for $\Gamma$. $Z(\Gamma^+)$ should lie in the strictly positive upper-half plane and $\Omega(\gamma)=0$ if $\gamma\not\in\Gamma^+$. 
  The solution is a piecewise continuous map $\Psi$ with values in $\Aut\C[\Gamma^+]$, and it is expressed as a sum of iterated integrals indexed by rooted trees whose vertices are labelled by elements of $\Gamma^+$. The proof involves 1) solving a fixed point integral problem, and 2) expanding terms of type $\log( 1 - x_{\gamma} )$ formally as $\sum_{k\geq 1} \frac{x_{k\gamma}}{k}$, in order to apply the Plemelj's theorem, which is a standard tool in the theory of Riemann-Hilbert problems.  
 % determined by rooted trees with vertices decorated by elements of $\Gamma$. 
The arguments extend to a generic and convergent BPS structure, provided that there exists a strictly convex cone $\Gamma^+\subset\Gamma$ such that $\Omega(\gamma)= 0$ for $\gamma\in\Gamma\setminus\left(\Gamma^+\cup-\Gamma^+\right)$, and  $Z(\Gamma^+)$ is contained in a strictly convex cone in $\C^*$. Let $\gamma_1,\dots,\gamma_n$ be a basis for such $\Gamma^+$. In this case $\stokes(\ell)$ and $\Psi$ are replaced by maps $\stokes_{\bfs}(\ell)$ and $\Psi_{\bfs}$ with values in  the completion $\C[\Gamma]\pow{\bfs}$ of $\C[\Gamma]$ with respect to a vector $\bfs=(s_1,\dots,s_n)$ of extra variables attached to each characters $x_{\gamma_i}$, $i=1,\dots,n$. The reader can see  \cite[{Section 4}]{FGS} for the explicit formulae and the proof in the \lq\lq positive\rq\rq\ case and compare with \cite[{Section\,4}]{BS} for the notation in $\C[\Gamma]\pow{\bfs}$. Note also that, due to a different sign conventions, the central charge in \cite{FGS} is $-Z$.

In the case of an uncoupled BPS structure, integrals in the solution are not iterated and $\Psi(t)_{|\bfs=\mathbf{1}}(x_{\beta})$ is given by the sum
	\begin{equation*}
	x_{\beta}\cdot \exp\left( \frac{Z(\beta)}{t}\right)  \exp\left( \frap \sum_{\gamma\in\Gamma\setminus\{0\}} \bra \beta, \gamma\ket \dt(\gamma)x_{\gamma} \int_{Z(\gamma)\Rp}\frac{te^{Z(\gamma)/s}}{s(s-t)} d s \right),
	\end{equation*}
where 
	\[\dt(\gamma):=\sum_{\substack{h\in\N\setminus\{0\} \\ \gamma/h\in\Gamma}}\frac{\Omega(\gamma/h)}{h^2}.\] One may deduce a solution to the (scalar) meromorphic Riemann-Hilbert problem \ref{meroRHpb} for uncoupled BPS structures satisfying the conditions above by evaluating $\Psi(t)_{|\bfs=1}(x_{\beta})$ at $\xi\in\torus$. Recall that $x_{\gamma}(\xi)=\xi(\gamma)\in\C$ and $h^{-2}\bra\beta,\gamma\ket = h^{-1}\bra\beta,\gamma/h\ket$. Formally, reindexing and reordering the double sum $$\sum_{\gamma\in\Gamma\setminus\{0\}}\, \sum_{h>0,\,\gamma/h\in\Gamma},$$ we have that $\Psi(t)_{|\bfs=\mathbf{1}}\left(x_{\beta}\right)$ applied to $\xi$ can be written as 
	\begin{equation}\label{formalsoluncoupled}
	\xi(\beta) e^{Z(\beta)/t} \exp \left(\frap \sum_{\gamma\text{ active}}\sum_{h\geq 0}\bra\beta,\gamma\ket \frac{\Omega(\gamma)}{h}\int_{Z(\gamma)\Rp}\frac{t}{s(s-t)} e^{hZ(\gamma)/s}\xi(\gamma)^h ds\right).
	\end{equation}
The formula \eqref{formalsoluncoupled} is a piecewise continuous map $\C^*\setminus\Sigma\to\C$ to be extended over the active rays to half-planes.  If we now want to compare it with the minimal solution $Y_{\beta,r}$ \eqref{sol_finite_uncoupled} we should  assume that the sum $-\sum \frac{1}{h}\left(e^{Z(\gamma)/s}\xi(\gamma)\right)^h$ converges to $\ln \left(1 - e^{Z(\gamma)/s}\xi(\gamma)\right)$. Changing variable $s\mapsto Z(\gamma)/s$ and integrating by parts, one sees that \eqref{formalsoluncoupled} coincides with \eqref{sol_finite_uncoupled}. The equality is a consequence of Theorem \ref{thmSolInt}.

\section{The Hamiltonian vector field for the doubled $A_1$ BPS structure}\label{sec_hamilt}

Riemann-Hilbert problems are related with the theory of irregular differential equations as inverse problems. Let $U,V\in\mathfrak{g}=\gl_n(\C)$ and 
	\begin{equation}\label{*}
	\nabla=d-A dt, \quad A=\frac{U}{t^2}+\frac{V}{t},
	\end{equation}
be a meromorphic connection with irregular pole at the origin and logarithmic pole at infinity. For every direction $r$ which is not a Stokes ray with a non-trivial Stokes factor in $\gl_n(\C)$, a fundamental solution lives in the half-plane centred in $r$, undergoing a discontinuity given by a Stokes factor as $r$ crosses a Stokes ray, \cite{stokesth}. If the solution $Y$ to the corresponding RH problem can be inverted, we compute $A=dY\cdot Y^{-1}$. 

In this section we describe a similar picture. If $\torus$ is a symplectic torus, we look at the solution to the $\Aut(\torus)$--valued Riemann-Hilbert problem as fundamental solution to a meromorphic connection of type \eqref{*} with $U,V$ in a different Lie algebra. More precisely, they are symplectic vector fields over the torus. 
\begin{rmk}A family of connections of the form \eqref{*} plays an important role in the theory of Frobenius manifolds. This points of view, in the context of BPS structures, is going to be developed and formalised by Bridgeland in the work in progress \cite{bridg_RHpb2}. 
\end{rmk}

For simplicity, we assume that $\Pi$ is a lattice of finite rank $m$ with trivial pairing 
	\[\bra-,-\ket\equiv 0,\]
and $(\Pi,Z,\Omega)$ a BPS structure satisfying all the conditions of Definition \ref{conditions}. Denote by $\Gamma$ the doubled lattice $\Pi\oplus\Pi^\vee$ with pairing \eqref{doubled_pairing} and by $\torus$ the twisted torus associated with $\Gamma$. In Theorem \ref{thm_F} below we compute the connection on the $\Aut(\torus)$-principal bundle corresponding to the doubled BPS data, having generalised monodromy given by
	\[\left\{\left(\ell_{\gamma}, \stokes(\ell_{\gamma})\right) \mid \Omega(\gamma)\neq 0\right\}.\] 
This has the form
	\begin{equation}\label{eqq}
	d - \left(\frac{Z}{t^2} + \frac{\Ham_F}{t}\right)dt,
	\end{equation}
where $Z$ is a vector field corresponding to the central charges, and $\Ham_F$ is the Hamiltonian vector field of a function $F:\torus\to\C$ depending on $Z$ and on the BPS spectrum, due to the Stokes factors. The work of Bridgeland and Toledano-Laredo \cite{BTL1} suggests that the residue part of the connection associated to a BPS structure by mean of such a Riemann-Hilbert problem should be seen as the carrier of the information of the BPS spectrum. We define the function $$F_\Omega=\frac{1}{2\pi i} \sum_{\gamma\in\Pi\setminus\{0\}} \Omega(\gamma)\Li_2\left(x_{\gamma}\right)$$
on the twisted torus and we interpret it as a generating function for the $\Omega$--invariants.

\begin{thm}\label{thm_F}
A solution $\Psi$ to the $\Aut\torus$-valued RH problem attached to $(\Gamma,Z,\Omega)$ 
is a flat section of the meromorphic connection 
	\begin{equation*}
	\nabla_{\Omega} = d - \left( \frac{Z}{t^2} + \frac{\Ham_{F_{\Omega}}}{t}\right)
	\end{equation*}
on the trivial $\Aut(\torus)$--bundle over $\C\PP^1$. 
\end{thm}
Before proving the Theorem, recall from Definition \ref{def_double} that the doubled BPS structure $(\Gamma,Z,\Omega)$ is a finite convergent uncoupled BPS structure of rank $2m$. In particular 
\begin{itemize}
\item $Z\in\Hom(\Gamma,\C)$, $Z_{|\Pi^\vee}\equiv 0$,
\item $\Omega:\Gamma\to\Z$, $\Omega_{|\Pi^\vee}\equiv 0$,
\end{itemize}
If $\{\gamma_1,\dots, \gamma_m\}$ is a basis of $\Pi$, then $\torus$ inherits logarithmic coordinates 
	\begin{equation*}
	\th_j:=\th(\gamma_j),\quad \th_j^{\vee}:=\th(\gamma_j^{\vee}),\quad j=1,\dots, m,
	\end{equation*}
and comes equipped with the symplectic form $\omega=-\sum_{j=1}^{m} d\th_j \wedge d\th_j^{\vee}$. The meromorphic Riemann-Hilbert problem \ref{scalarRHpb} for $(\Gamma,Z,\Omega)$ admits solutions $Y_{\beta,r}$ defined in \eqref{sol_finite_uncoupled}. These induce a solution $\Psi:\C^*\setminus\Sigma \to \Aut(\torus)$ to the $\Aut(\torus)$-valued problem \ref{autTvaluedRHpb}
	\[
	\left(\Psi(t)(\xi)\right)(\beta) = e^{-Z(\beta)/t}\cdot Y_{\beta,r}\cdot \xi(\beta).
	\]

\begin{proof}[Proof of Theorem \ref{thm_F}]
We first compute the connection \eqref{eqq} for the doubled $A_1$ BPS structure of Example \ref{exm_A1}, from which we also borrow the notation. The proof generalises to the doubled of any uncoupled BPS structure. First observe that the solution $Y_{\pm}(t)$ to Problem \ref{RHA1} satisfies
	\begin{equation}\label{gradx}
	t \deld{t} \log Y_{\pm}(t) - \frac{z}{t}\deld{\theta}\log Y_{\pm}(t) = -\frac{1}{2\pi i} \left(\theta-\pi i\right).
	\end{equation}
This follows by direct computation from Lemma \ref{lem_grad} and the definition of $Y_{\pm}(t)$. Indeed we have
	\begin{equation}\label{xpmXX}
	Y_-(t)=\XXp\left( \theta-2\pi i ,\, \frac{z}{t}\right), \quad
	Y_+(t)=\XXp\left( -\theta,\, - \frac{z}{t}\right)^{- 1}
	\end{equation}
in some domain.

$Y_{\pm}$ induces a map $\Psi_{\pm}:\C\setminus\ell_{\pm}\to \Aut\torus$ defined by
	\begin{align*}
	\Psi_{\pm}(t)(\xi)(\alpha) &= e^{-Z(\alpha)/t}\cdot \xi(\alpha), \\
	\Psi_{\pm}(t)(\xi)(\alpha^\vee) &= e^{-Z(\alpha^\vee)/t}\cdot Y_{\pm}(t)\cdot \xi(\alpha^\vee).
	\end{align*}
As $Z(\alpha^\vee)=0$, $Z(\alpha)=:z$, in logarithmic coordinates $\Psi_{\pm}$ reads
	\[ \Psi_{\pm}(t): \begin{cases} \theta \mapsto -\frac{z}{t} + \theta \\
	\theta^\vee \mapsto \theta^\vee + \log Y_{\pm}(t)\end{cases},
	\]
with $t$-derivative components
	\begin{equation}\label{dPhidt}
	\begin{pmatrix} \frac{z}{t^2}\\
	\frac{d}{d t}\log Y_{\pm}(t)
	\end{pmatrix},
	\end{equation}
and Jacobian for fixed $t$
	\[ Jac\,\Psi_{\pm} = \begin{pmatrix} 1 & \deld{\th} \log Y_{\pm} \\ 0 & 1 \end{pmatrix}.
	\]
By $RH_1$, $RH_2$, $RH_3$, $\Psi_{\pm}$ solves a differential equation with leading term $z/t^2$, Stokes factors $\stokes(\pm\ell)$, and at most logarithmic pole at infinity. We may write such an equation as 
	\begin{equation}\label{odeX}
	\frac{d \Psi_{\pm}(t)}{d t} = \left( \frac{Z}{t^2} + \frac{\Ham_F}{t}\right) \Psi_{\pm}(t),
	\end{equation} 
where $Z = z \frac{\del}{\del \th}$ and $\Ham_F$  is the hamiltonian vector field 
	\[
	\Ham_F = - \frac{\del F}{\del \thdual}\frac{\del}{\del \th} + \frac{\del F}{\del \th}\frac{\del}{\del \thdual}
	\]
of a function $F=F(\th,\thdual):\torus\to\C$. Knowing $Jac\,\Psi_{\pm}$, the right hand side of \eqref{odeX} has components
	\begin{equation}\label{rhs}
	\begin{pmatrix} \frac{1}{t^2}z - \frac{1}{t}\frac{\del F}{\del \thdual} \\
	\frac{z}{t^2} \frac{\del \log Y_{\pm}(t)}{\del \th} - \frac{1}{t} \frac{\del F}{\del \thdual}\frac{\del \log Y_{\pm}}{\del \th} + \frac{1}{t}\frac{\del F}{\del \th} \end{pmatrix}.
	\end{equation}
Equalling \eqref{dPhidt} and \eqref{rhs}, we deduce that $\frac{\del F}{\del \thdual} = 0$, and 
	\begin{equation*}
	-\frac{\del F}{\del \th} = -t\frac{\del}{\del t} \log Y_{\pm} + \frac{z}{t} \frac{\del}{\del \th} \log Y_{\pm}.
	\end{equation*}
By \eqref{gradx}, this means $-\frac{\del F}{\del \th} = \frap\left(\th-\pi i\right)$, or $F(\th,\thdual)= -\frac{\th^2}{4\pi i} + \frac{\th}{2} + const$.

The very same computations can be performed in higher dimension. For any generic ray $r$, and $j=1,\dots, m$, equation \eqref{gradx} generalises to 
	\[
	t \deld{t} \log Y_{\gamma_j^\vee,r}(t) - \sum_{k=1}^m \frac{z_k}{t}\deld{\theta_k}\log Y_{\gamma_j^\vee,r}(t) = \frac{1}{2\pi i} \sum_{\gamma\in\Gactiver} \coef{\gamma}{j}\bra\gamma_j^\vee,\gamma\ket \Omega(\gamma)\left(\theta(\gamma)-\pi i\right),
	\]
where $z_k:=Z(\gamma_k)$, $\coef{\gamma}{j}$ denotes the $j$-th component of $\gamma$ with respect to the chosen basis, and $\Gactiver$ was defined in \eqref{Gactiver}. Defininig the vector fields
	\begin{gather*}
	Z = \sum_j Z(\gamma_j) \frac{\del}{\del \th_j} + \sum_j Z(\gamma_j^{\vee}) \frac{\del}{\del \thdual_j} =\sum_j Z(\gamma_j) \frac{\del}{\del \th_j} , \\
	\Ham_F = - \sum_j \frac{\del F}{\del \thdual_j}\frac{\del}{\del \th_j} + \sum_j \frac{\del F}{\del \th_j}\frac{\del}{\del \thdual_j},
	\end{gather*}
one obtains that 
	\[-\frac{\del F}{\del \th_j} = \frap \sum_{\gamma\in\Gactiver} \coef{\gamma}{j}\Omega(\gamma) \left( \th(\gamma)-\pi\ii\right),	\quad j=1,\dots,m,
	\]
hence $F=- \sum_{\gamma\in\Gactiver} \Omega(\gamma) \left( \frac{\th(\gamma)^2}{4\pi\ii}-\frac{\th(\gamma)}{2}\right) + const$. Note that $\Gactive\subset\Gamma$, defined in \eqref{Gactive}, as well as $\Gactiver$ select half of the points of $\Gamma\setminus\{0\}$.

Recall now that $\Li_2\left(x_{\gamma}\right)=\sum_{k\geq 0}\frac{x_{k\gamma}}{k^2}$ and, for $\xi\in\torus$, $x_{\gamma}(\xi)=\xi(\gamma)=e^{\theta(\gamma)}$ in logarithmic coordinates. For $0\leq \Im\theta<2\pi$, the dilogarithm satisfies
	\[\Li_2(e^{\th})+\Li_2(e^{-\th})= -\frac{(2\pi\ii)^2}{2}B_2\left(\frac{\th}{2\pi\ii}\right),\]
where $B_2(x)=x^2-x+\frac{1}{6}$ denotes the second Bernoulli polynomial. Then the function $F_\Omega=\frac{1}{2\pi i} \sum_{\gamma\in\Pi\setminus\{0\}} \Omega(\gamma)\Li_2\left(x_{\gamma}\right)$ applied to $\xi\in\torus$ is
	\[
	F_{\Omega}(\xi)= -\sum_{\gamma\in\Gactive\setminus\{0\}}\Omega(\gamma)\left(\frac{\th(\gamma)^2}{4\pi\ii}-\frac{\th(\gamma)}{2}+\frac{1}{12}\right)
	\]
and has the required form.
\end{proof}

\bibliographystyle{plain}

\begin{thebibliography}{10}

\bibitem{stokesth}
Werner Balser, Wolfgang~B.\ Jurkat, and Donald~A.\ Lutz.
\newblock Birkhoff invariants and Stokes multipliers for meromorphic linear differential equations.
\newblock {\em Journal of Mathematical Analysis and Applications}, 71.1, 1979.

\bibitem{qRH} Anna Barbieri, Tom Bridgeland, and Jacopo Stoppa, \newblock A quantized Riemann-Hilbert problem in {D}onaldson-{T}homas theory, 
\newblock {arXiv preprint \texttt{arXiv:1905.00748}}, 2019

\bibitem{BS}
Anna Barbieri and Jacopo Stoppa.
\newblock Frobenius type and {CV} structures for {D}onaldson-{T}homas theory and a convergence property.
\newblock {\em Communications in {A}nalysis and {G}eometry}, to appear, 2017.

\bibitem{bridg_RHpb}
Tom Bridgeland.
\newblock Riemann-{H}ilbert problems from {D}onaldson-{T}homas theory.
\newblock {\em Inventiones Mathematicae}, 1-56 (2016).

\bibitem{bridg_RHpb2}
Tom Bridgeland.
\newblock Riemann-{H}ilbert problems from {D}onaldson-{T}homas theory {II}, {I}n preparation.

\bibitem{BTL1}
Tom Bridgeland and Valerio Toledano-Laredo.
\newblock {Stability conditions and Stokes factors.}
\newblock {\em Inventiones {M}athematicae}, 187(1):61--98, 2012.

\bibitem{FGS}
Sara~Angela Filippini, Mario Garcia-Fernandez, and Jacopo Stoppa.
\newblock Stability data, irregular connections and tropical curves.
\newblock {\em J.\ Sel.\ Math.\ New Series} \textbf{23} (2017).

\bibitem{fokas_painleve}
Athanassios~S.\ Fokas, Alexander~R.\ Its, Andrei~A.\ Kapaev, and Victor~Y.\ Novokshenov.
\newblock {\em Painlev\'e trascendents: the Riemann-Hilbert approach}.
\newblock AMS, Mathematical Surveys and Monographs, vol. 128, 2006.

\bibitem{gaiotto}
Davide Gaiotto.
\newblock {Opers and TBA}.
\newblock {arXiv preprint \texttt{arXiv:1403.6137}}, 2014.

\bibitem{GMN}
Davide Gaiotto, Gregory~W.\ Moore, and Andrew Neitzke.
\newblock Four-dimensional wall-crossing via three-dimensional field theory.
\newblock {\em Communications in Mathematical Physics}, 299(1):163--224, 2010.

\bibitem{gakhov_boundary_pbs}
Fedor~D.\ Gakhov.
\newblock {\em Boundary Value Problems}.
\newblock Pergamon Press, 1966.

%\bibitem{joyce}
%Dominic Joyce.
%\newblock {Holomorphic generatig functions for invariants counting coherent sheaves on Calabi-Yau 3-folds.}
%\newblock {\em Geometry \& Topology}, (11(2)):667--725, 2006.

\bibitem{KS}
Maxim Kontsevich and Yan Soibelman.
\newblock {Stability structures, motivic Donaldson-Thomas invariants and cluster transformations}, arXiv preprint \texttt{arXiv:0811.2435}, 2008.

\bibitem{magnus}
Wilhelm Magnus, Fritz Oberhettinger, and Raj~P.\ Soni.
\newblock {\em Formulas and theorems for the special functions of Mathematical Physics}.
\newblock Springer-Verlag Berlin, 1966.

\bibitem{muske}
Nikoloz~I.\ Muskhelishvili.
\newblock {\em Singular Integral Equations: Boundary Problems of Function Theory and Their Application to Mathematical Physics}.
\newblock Dover Publications, 1946.

\bibitem{SS}
Jacopo Scalise and Jacopo Stoppa.
\newblock {Variations of BPS structure and a large rank limit.}
\newblock {arXiv preprint \texttt{arXiv:1705.08820}}, 2017.

\bibitem{whittaker}
Edmund~T.\ Whittaker and George~N.\ Watson.
\newblock {\em A course of modern analysis}.
\newblock Cambridge University Press, 1950.

\end{thebibliography}

\end{document}